\documentclass[a4paper,10pt]{amsart}
\usepackage{amsmath,amsthm,amsfonts,amssymb,calrsfs}

\usepackage[all]{xy}
\usepackage{tikz-cd, tikz}

\usepackage{macro-math} 

\usepackage{graphicx}
\usepackage{yfonts}
\usepackage{xcolor}

\usepackage{array}
\usepackage{multirow}
\usepackage{caption}

\usepackage[hyperindex=true, colorlinks=true, linkcolor=blue, filecolor=black,
citecolor=blue, urlcolor=blue, bookmarks=true, bookmarksopen=true,
bookmarksnumbered=true, bookmarksopenlevel=2, pdfborder={0 0 0}]{hyperref}
\usepackage{cleveref}

\theoremstyle{plain}{
    \newtheorem{theorem}{Theorem}[section]
    \newtheorem{lem}[theorem]{Lemma}
    
    \newtheorem{prop}[theorem]{Proposition}
    
}

\theoremstyle{definition}{
    \newtheorem{defn}[theorem]{Definition}
    \newtheorem{defn-thm}[theorem]{Definition-Theorem}

}

\theoremstyle{remark}{
    \newtheorem{rem}[theorem]{Remark}

}

\begin{document}

\title[Log smooth toric del Pezzo pairs]{Classification of log smooth toric del
Pezzo pairs}

\author[A. Napame]{Achim NAPAME}

\address{Univ Brest, UMR CNRS 6205, Laboratoire de Mathématiques de Bretagne
Atlantique, France}

\email{achim.napame@univ-brest.fr}


\begin{abstract}
We give a description of all log-Fano pairs $(X, \, D)$ where $X$ is a smooth toric
surface and $D$ a reduced simple normal crossing divisor such that the pair $(X, \, D)$
is equivariant.
\end{abstract}

\maketitle

\section{Introduction}

The Enriques-Kodaira classification gives a classification of complex compact surfaces
by using their Kodaira dimension. Nonsingular projective minimal surfaces with
Kodaira dimension $-\infty$ have an important position in this problem of classification,
they correspond in the MMP-terminology to {\it Mori's fiber spaces}
(see \cite[Theorem 1.5.5]{Mat02}).
Del Pezzo surfaces form a class of surfaces with kodaira dimension $-\infty$.

A {\it toric variety} $X$ of dimension $n$ is an irreducible variety that contains a
torus $T \simeq (\Cm^{\ast})^n$ as a dense open subset, together with an action of
$T$ on $X$ that extends the natural action of $T$ on itself.
We say that a toric surface is {\it log del Pezzo} if it has at worst log-terminal
singularities and its anticanonical divisor is a $\Q$-Cartier ample divisor.
According to \cite[Proposition 4.2]{Dais17}, there is a one-to-one correspondence between
toric log del Pezzo surfaces and LDP-polygons which are convex lattice polygons
containing the origin in their interior and having only primitive vertices. The
classification of toric log del Pezzo surfaces is an open problem. There are results on
the classification according to the index of the anticanonical divisor (see
\cite[Section 8.3]{CLS} for index one and \cite{KKN10} for index two).

Motivated by {\it Iitaka's philosophy}, in this paper we are interested by the question
of classification of log smooth toric del Pezzo pairs.
A pair $(X, \, D)$ is called a {\it log smooth del Pezzo pair} if $X$ is a smooth
surface and $D$ a simple normal crossing divisor such that $-(K_X + D)$ is ample.
Moreover, we say that $(X, \, D)$ is {\it toric} if $X$ is toric and the sheaf
$T_{X}(- \log D)$ is an equivariant sheaf which means that $D$ is an invariant divisor
under the action of the torus $T$ of $X$.

\begin{theorem}\label{art201}
Let $X$ be a complete smooth toric surface such that $\rk \Pic(X) \geq 3$. For any
reduced invariant divisor $D$ of $X$, the pair $(X, \, D)$ is not toric log del Pezzo.
\end{theorem}

According to this Theorem, to find the classification of log smooth toric del Pezzo
pairs, it suffices to consider complete smooth toric surfaces with Picard rank one
or two.
The variety $\Pm^2 = \Cm \Pm^2$ is the unique smooth toric surface of Picard rank one.
We denote by $[x_0: x_1: x_2]$ the homogeneous coordinates of $\Pm^2$ and $D_i$ the
divisor defined by $x_i = 0$.

\begin{theorem}\label{art202}
If $X = \Pm^2$, then the log smooth pair $(X, \, D)$ is toric log del Pezzo if and only
if $~D \in \{D_0, \, D_1, \, D_2 \} \cup \{D_0 + D_1, \, D_0 + D_2, \, D_1 + D_2 \} ~.$
\end{theorem}

Every smooth toric surface of Picard rank two is of the form
$X = \Pm \left( \Oc_{\Pm^1} \oplus \Oc_{\Pm^1}(r) \right)$ with $r \in \N$
(we refer to \cite[Theorem 7.3.7]{CLS} and \cite{Kle88}). We denote
by $\pi : X \imp \Pm^1$ the projection on the base $\Pm^1$.
Let $E_0 = \{0 \}$ and $E_{\infty} = \{\infty\}$ be two divisors of
$\Pm^1 = \Cm \cup \{\infty \}$ (the identification come from
$[x_0: x_1] \giv x_1/x_0$).
Let $\Fc_0$ and $\Fc_1$ be two rank one vector bundles over $\Pm^1$ whose sheaf of
sections are respectively given by $\Oc_{\Pm^1}$ and $\Oc_{\Pm^1}(r)$. We denote by
$s_i$ a section of $\Fc_i$.
The irreducible $T$-invariant divisors of $X$ are given by
$$
D_0 = \{s_0 = 0 \} ~, \quad D_1 = \pi^{-1}(E_0) ~, \quad D_2 = \{s_1 = 0 \} \quad
\text{and} \quad D_3 = \pi^{-1}(E_{\infty}) ~.
$$
If $D = \sum_{i \in I}{D_i}$ with $I \subset \{0, \, 1, \, 2, \, 3 \}$ such that
$\card(I) \geq 3$, then $-(K_X + D)$ is not ample by Proposition \ref{art216}.
He have here the classification of toric log del Pezzo pairs for toric surfaces with
Picard rank two.

\begin{theorem}\label{art203}
Let $X = \Pm \left( \Oc_{\Pm^1} \oplus \Oc_{\Pm^1}(r) \right)$ with $r \in \N$. Then :
\begin{enumerate}
\item
$-K_X$ or $-(K_X + D_0)$ are ample if and only if $r \in \{0, \, 1 \}$.
\item
If $D \in \{D_1, \, D_3, \, D_0 + D_1, \, D_0 + D_3 \}$, $-(K_X + D)$ is ample if and
only if $r = 0$.
\item
If $D \in \{D_2, \, D_2 + D_1, \, D_2 + D_3 \}$, $-(K_X + D)$ is ample if and only if
$r \in \N$.
\item
if $D \in \{D_0 + D_2, \, D_1 + D_3 \}$, $-(K_X+D)$ is not ample for any $r \in \N$.
\end{enumerate}
\end{theorem}

In the first part of this paper, we will recall some properties about toric varieties
and we will give the characterization of ample divisors on toric surfaces. In the second
part (Section \ref{artsec202}), we will describe the family of log smooth toric del
Pezzo pairs.

\subsection*{Acknowledgments}
I would like to thank my advisor Carl \textsc{Tipler} for our discussions on this
subject.

\section{Ample divisors on toric surfaces}

\subsection{Toric varieties}
Let $N$ be a rank $n$ lattice and $M = \Hom_{\Z}(N, \, \Z)$ be its dual with pairing
$\langle \, \cdot \, , \, \cdot \, \rangle : M \times N \imp \Z$. The lattice $N$ is the
lattice of one-parameter subgroups of $N \otimes_{\Z} \Cm^{\ast}$.
For $\K = \R ~ \text{or} ~ \Cm$, we define $N_{\K} = N \otimes_{\Z} \K$ and
$M_{\K} = M \otimes_{\Z} \K$.
A fan $\Sigma$ in $N_{\R}$ is a set of rational strongly convex polyhedral cone
in $N_{\R}$ such that:
\begin{itemize}
\item
Each face of a cone in $\Sigma$ is also a cone in $\Sigma$;
\item
The intersection of two cones in $\Sigma$ is a face of each.
\end{itemize}

\begin{defn}[{\cite[Definition 1.2.16]{CLS}}]
A cone $\sigma$ in $N_{\R}$ is {\it smooth} if its minimal generators form part of a
$\Z$-basis of $N$.
\end{defn}

By \cite[Definition 3.1.18]{CLS}, a fan $\Sigma$ is {\it smooth} if every cone $\sigma$
in $\Sigma$ is smooth. A fan $\Sigma$ is {\it complete} if
$$
\bigcup_{\sigma \in \Sigma}{\sigma} = N_{\R} ~.
$$
For $\sigma \in \Sigma$, let $U_{\sigma} = \Spec( \Cm[S_{\sigma}])$ where
$\Cm[S_{\sigma}]$ is the semi-group algebra of
$$
S_{\sigma} = \sigma^{\vee} \cap M = \{ m \in M \, : \, \langle m, \, u \rangle \geq 0
~ \text{for all}~ u \in \sigma \} ~~.
$$
If $\sigma, \, \sigma' \in \Sigma$, we have $U_{\sigma} \cap U_{\sigma'} =
U_{\sigma \cap \sigma'}$.
We denote by $X_{\Sigma}$ the toric variety associated to a fan $\Sigma\,$; the variety
$X_{\Sigma}$ is obtained by gluing affine charts $(U_{\sigma})_{\sigma \in \Sigma} \,$.
\\[0.25cm]
The variety $X_{\Sigma}$ is normal and its torus is $T = N \otimes_{\Z} \Cm^{\ast}$.
Let $\Sigma(k)$ be the set of $k$-dimensional cones of $\Sigma$. There is a bijective
correspondence between cones $\sigma \in \Sigma$ and $T$-orbits $O(\sigma)$ in
$X_{\Sigma} \,$; moreover we have $\dim O(\sigma) = \dim(X_{\Sigma}) - \dim(\sigma)$.
\\
For any ray $\rho \in \Sigma(1)$, there is a Weil divisor $D_{\rho}$ defined as the
Zariski closure of $O(\rho)$ in $X_{\Sigma}\,$.
We denote by $u_{\rho} \in N$ the minimal generator of $\rho \in \Sigma(1)$.
By \cite[Theorem 4.1.3]{CLS}, if $N_{\R} = \Span(u_{\rho} : \rho \in \Sigma(1))$ then
\begin{equation}\label{art204}
\card(\,  \Sigma(1) \,) = \dim(X_{\Sigma}) + \rk \Cl(X_{\Sigma}) ~.
\end{equation}
Let $X$ be the toric variety associated to a fan $\Sigma$. We say that $X$
is {\it smooth} (resp. {\it complete}) if and only if $\Sigma$ is smooth
(resp. complete). We know the expression of the canonical divisor $K_{X}$.

\begin{theorem}[{\cite[Theorem 8.2.3]{CLS}}]\label{art205}
The canonical divisor of $X$ is the torus invariant divisor
$$
K_{X} = - \sum_{\rho \in \Sigma(1)}{D_{\rho}} ~.
$$
\end{theorem}

A Cartier divisor of $X$ is ample, if it satisfies the toric Kleiman Criterion.

\begin{theorem}[Toric Kleiman Criterion, {\cite[Theorem 6.3.13]{CLS}}]\label{art206}
Let $D$ be a Cartier divisor on a complete toric variety $X$. Then $D$ is ample if
and only if $D \cdot C > 0$ for all torus-invariant irreducible curves $C \subset X$.
\end{theorem}

\subsection{Description of complete toric surfaces}\label{artsec201}
We assume that $N = \Z^2$, $M = \Z^2$ and the pairing
$\langle \, \cdot \, , \, \cdot \, \rangle : M \times N \imp \Z$ is the dot product.
Let $\Sigma$ be a smooth complete fan in $\R^2$ and $X$ the toric surface associated
to $\Sigma$. We denote by $T$ the torus of $X$. There is a family of primitive vectors
$\{ u_i \in N : ~ 0 \leq i \leq n-1 \}$
with $n \geq 3$ such that
$$
\Sigma = \{0 \} \cup \left\lbrace \Cone(u_i) : \, 0 \leq i \leq n-1 \right\rbrace \cup
\left\lbrace \Cone(u_i, \, u_{i+1}) : \, 0 \leq i \leq n-1 \right\rbrace
$$
where $u_{n} = u_0$.
For all $i \in \{0, \ldots, \, n-1\}$, we set
$\rho_i = \Cone(u_i)$, $\sigma_i = \Cone(u_i, \, u_{i+1})$ and $D_i$ the divisor
corresponding to the ray $\rho_i$. As $\Sigma$ is smooth,
we can assume that $\det(u_i, \, u_{i+1}) = \pm 1$. From now on, we will assume that
$\det(u_i, \, u_{i+1}) = 1$.
\\
For all $i \in \{0, \ldots, \, n-1 \}$, there is
$\gamma_i \in \Z$ such that
\begin{equation}\label{art207}
u_{i-1} - \gamma_i \, u_i + u_{i+1} = 0 ~.
\end{equation}
The number $\gamma_i$ is $\det(u_{i-1}, \, u_{i+1})$.

\subsection{Ample divisors of $X$}
In this part, if $L = \sum_{i}{a_i D_i}$ is a Cartier divisor of $X$, we will give
a condition on the numbers $a_i$ which ensures that $L$ is ample.\\
By \cite[Proposition 6.4.4]{CLS}, for all $i \in \{0, \ldots, \, n-1 \}$, we have
\begin{equation}\label{art208}
\left\lbrace
\begin{array}{ll}
D_i \cdot D_i = - \gamma_i & \\
D_k \cdot D_i = 1 & \text{if} ~ k \in \{i-1, \, i+1 \} \\
D_k \cdot D_i = 0 & \text{if} ~ k \notin \{i-1, \, i, \, i+1 \}
\end{array}
\right.~~.
\end{equation}

\begin{prop}\label{art209}
A Cartier divisor $L = \sum_{i}{a_i \, D_i}$ is ample if and only if
for all $i \in \{0, \ldots, \, n-1 \}$, $~ a_{i+1} + a_{i-1} - \gamma_{i} \, a_i > 0~$.
\end{prop}

\begin{proof}
The set of torus-invariant irreducible curve of $X$ is
$\{ D_i : 0 \leq i \leq n-1 \}$. By (\ref{art208}), we have
\begin{equation}\label{art210}
L \cdot D_i = a_{i-1} + a_{i+1} - \gamma_i \, a_i ~.
\end{equation}
By using the toric Kleiman Criterion (Theorem \ref{art206}), we deduce that $L$ is
ample if and only if $~ a_{i+1} + a_{i-1} - \gamma_{i} \, a_i > 0~$ for all
$i \in \{0, \ldots, \, n-1 \}$.
\end{proof}

Let $L = \sum_{i}{a_i \, D_i}$ be a Cartier divisor of $X$. The polytope
corresponding to $L$ is given by
\begin{equation}\label{art211}
P = \{ m \in \Z^2 \, : \, \langle m \, , \, u_i \rangle \geq -a_i ~ \text{for} ~
i \in \{0, \ldots, \, n-1 \} \}
\end{equation}
and the facet of $P$ perpendicular to the vector $u_i$ is given by
\begin{equation}\label{art212}
P_i = \{m \in \Z^2 \, : \, \langle m \, , \, u_i \rangle = -a_i \} \cap P ~.
\end{equation}
The polytope $P$ is a polygon; hence any vertex of $P$ is precisely the intersection
of two facets. If for all $i \in \{0, \ldots, \, n-1 \}$, $P_i \neq \vide$, then
$P_i \cap P_{i+1}$ contains only one point; we set $\{ m_i \} = P_{i} \cap P_{i+1}$, we
have
$$
P = \Conv(m_i : \, 0 \leq i \leq n-1)
$$
and $P_i$ is the edge having $m_{i-1}$ and $m_i$ for extremities.

\begin{rem}
The point $m_i$ is the solution of the equations
$\langle m_i, \, u_i \rangle = -a_i$ and $\langle m_i, \, u_{i+1} \rangle = -a_{i+1}$.
\end{rem}

We recall that a lattice $M$ defines a measure $\nu$ on $M_{\R}$ as the pull-back of
the Haar measure on $M_{\R} / M$. The measure $\nu$ is translation invariant and
satisfies $\nu(M_{\R}/ M) = 1$.
Let $\vol(P_i)$ be the volume of $P_i$ with respect to the measure determined by the
lattice
$$
\Lambda = M \cap \{m \in \Z^2 \, : \, \langle m \, , \, u_i \rangle = -a_i \}
$$
in the affine span of $\Lambda$.

\begin{prop}\label{art213}
Let $P$ be the polytope corresponding to a Cartier divisor $L= \sum_{i}{a_i \, D_i}$
such that for all $i \in \{0, \ldots, \, n-1 \}$, $P_i \neq \vide$. We have
$$
\vol(P_i) = |a_{i+1} + a_{i-1} - \gamma_i \, a_i| ~.
$$
\end{prop}

We first observe that :
\begin{lem}\label{art214}
If $x =(x_1, \, x_2)$ and $y = (y_1, \, y_2)$ are two point of $\Z^2$, then
$$
\card \left( \, \{ tx + (1-t)y : \, t \in [0 \, ; \, 1] \} \cap \Z^2 \, \right) - 1=
\gcd( |x_1 - y_1|, \, |x_2 - y_2| ) ~.
$$
\end{lem}

\begin{proof}
We can assume that $(x_1, \, x_2) = (0, \, 0)$. Let
$A = \{ t \, y : \, t \in [0 \, ; \, 1] \} \cap \Z^2$. If $y_1 = 0$, then
$\card(A) = y_2 + 1$ and when $y_2 = 0$, $\card(A) = y_1 + 1$.
For the case $y_1 \neq 0$ and $y_2 \neq 0$, we can reduce the study to the case
where $y_1 >0$ and $y_2 >0$.
\\[0.25cm]
{\it First case : We assume that $\gcd(y_1, \, y_2) = 1$.} If $t=0$ or $t=1$,
then $t \, y \in A$. If there is $t \in \, ]0 \, ; \, 1[$ such that $t \, y \in A$, then
$t = \dfrac{p}{q}$ with $p, \, q \in \N$ such that $1 \leq p < q$.
As $\dfrac{p \, y_1}{q}, \, \dfrac{p \, y_2}{q} \in \N$, we deduce that $q$ divides
$y_1$ and $y_2$. This is in contradiction with the fact that $\gcd(y_1, \, y_2) = 1$.
Thus, we deduce that $\card(A) = 2$.
\\[0.25cm]
{\it Second case : $\gcd(y_1, \, y_2) \geq 2$}. We have $t \, y \in A$ if and only if
$t = \dfrac{k}{\gcd(y_1, \, y_2)}$ with
$k \in \{0, \, 1, \ldots, \, \gcd(y_1, \, y_2) \}$.
Thus, $\card(A) = \gcd(y_1, \, y_2) + 1$.
\end{proof}

\begin{proof}[Proof of Proposition \ref{art213}]
We have $\vol(P_i) = \card(P_i \cap \Z^2) -1$.
We write $u_i = \alpha_i \, e_1 + \beta_i \, e_2$ with $\alpha_i, \, \beta_i \in \Z$.
The equations
$\langle m_i, \, u_i \rangle = -a_i$ and $\langle m_i, \, u_{i+1} \rangle = -a_{i+1}$
give
$ m_i = \dbinom{a_{i+1} \, \beta_i - a_i \, \beta_{i+1}}{-a_{i+1} \, \alpha_i +
a_i \, \alpha_{i+1}} \,$.
We have
$$
\overrightarrow{m_{i-1} \, m_i} = \dbinom{\beta_i(a_{i+1} + a_{i-1}) -
a_i(\beta_{i+1} + \beta_{i-1})}{-\alpha_i(a_{i+1} + a_{i-1}) +
a_i(\alpha_{i-1} + \alpha_{i+1})} ~.
$$
The equalities $\det(u_{i-1}, \, u_i) = \det(u_i, \, u_{i+1}) = 1$ give
\begin{equation}\label{art215}
\alpha_i \left(\beta_{i-1} + \beta_{i+1} \right) = \beta_i \left( \alpha_{i-1} +
\alpha_{i+1} \right) ~~.
\end{equation}
{\it First case : We assume that $\alpha_i = 0$.}
As $\gcd(\alpha_i, \, \beta_i)= 1$, we deduce that $\alpha_{i-1} + \alpha_{i+1} = 0$
and by the relation $\det(u_{i-1}, \, u_i) = 1$, we have $\alpha_{i-1} \, \beta_i = 1$,
i.e $\alpha_{i-1} = \beta_i = \pm 1$. Hence,
$$
\gamma_i = \alpha_{i-1} \, \beta_{i+1} - \alpha_{i+1} \, \beta_{i-1} =
\beta_i \, \beta_{i+1} + \beta_i \, \beta_{i-1} = \beta_i \left( \beta_{i-1} +
\beta_{i+1} \right) ~.
$$
By Lemma \ref{art214}, we have
\begin{align*}
\vol(P_i) & =
\gcd( |\beta_i (a_{i+1} + a_{i-1}) - a_i(\beta_{i+1} + \beta_{i-1})|, \, 0)
\\ & =
\gcd( |(a_{i+1} + a_{i-1}) - a_i \, \beta_i(\beta_{i+1} + \beta_{i-1})|, \, 0)
\\ & =
\gcd( |a_{i+1} + a_{i-1} - a_i \, \gamma_i|, \, 0)
\\ & =
|a_{i+1} + a_{i-1} - \gamma_i \, a_i|
\end{align*}
The case $\beta_i =0$ is similar to the case $\alpha_i=0$.
\\[0.25cm]
{\it Second case : $\alpha_i \, \beta_i \neq 0$ and $\alpha_{i-1} + \alpha_{i+1} = 0$.}
By (\ref{art215}), we have $\beta_{i-1} + \beta_{i+1} = 0$.
Hence, $u_{i-1}= - u_{i+1}$ and $\gamma_i = 0$. Thus,
\begin{align*}
\vol(P_i) &= \gcd( |\beta_i(a_{i-1} + a_{i+1})|, \, |\alpha_i(a_{i-1} + a_{i+1})|)
\\ & =
|a_{i-1} + a_{i+1}| \gcd( |\beta_i|, \, |\alpha_i|)
\\ & =
|a_{i-1} + a_{i+1}|
\end{align*}
{\it Third case : $\alpha_i \, \beta_i \neq 0$ and $\alpha_{i-1} + \alpha_{i+1} \neq 0$.}
As $\gcd(\alpha_i, \, \beta_i) = 1$, by (\ref{art215}) we have
$\alpha_{i-1} + \alpha_{i+1} = \alpha_i$ and $\beta_{i-1} + \beta_{i+1} = \beta_i$.
Hence,
$$
\vol(P_i)=
\gcd( |\beta_i(a_{i+1} + a_{i-1} - a_i)|, \, |\alpha_i(a_{i+1} + a_{i-1} - a_i)|)
= |a_{i+1} + a_{i-1} - a_i |
$$
To finish the proof, we must show that $\gamma_i = 1$. We have
\begin{align*}
\gamma_i & = \alpha_{i-1} \, \beta_{i+1} - \alpha_{i+1} \, \beta_{i-1} \\ &=
\alpha_{i-1} (\beta_i - \beta_{i-1} ) - \beta_{i-1} (\alpha_i - \alpha_{i-1}) \\ &=
\alpha_{i-1} \, \beta_i - \beta_{i-1} \, \alpha_i \\ & = 1
\end{align*}
Hence, we get the proof.
\end{proof}

With this proof, we observe that : if $P$ is the polytope corresponding to the
polarized toric surface $(X, \, L)$, then $L \cdot D_i = \vol(P_i)$.
It is a particular case of the result of \cite[Section 11]{Dan78}.

\section{Smooth toric log del Pezzo pairs}\label{artsec202}

Let $X$ be a smooth complete toric surface associated to a fan $\Sigma$. If $D$ is a
reduced $T$-invariant Weil divisor of $X$, we will give a description of all pairs
$(X, \, D)$ such that $-(K_X + D)$ is ample. By (\ref{art204}), we have
$$
\card( \Sigma(1)) = 2 + \rk( \Pic(X)) ~.
$$
We keep the notations of the previous section.

\subsection{Analysis}
Let $\Delta$ be a subset of
$\{0, \ldots, \, n-1 \}$ and $\disp D = \sum_{i \in \Delta}{D_i} \,$.

\begin{prop}\label{art216}
Let $X$ be a complete smooth toric surface. If $\card( \Delta ) \geq 3$,
then $-(K_X + D)$ is not ample.
\end{prop}

\begin{proof}
By Theorem \ref{art205}, we have
$$
-(K_X + D) = \sum_{i \in \Delta'}{D_i}
$$
where $\Delta' = \{0, \ldots, \, n-1\} \setminus \Delta$.
Let $P$ be the polytope corresponding to $-(K_X + D)$.
We set $\partial P = P_0 \cup P_1 \cup \ldots \cup P_{n-1}$.
By (\ref{art212}), for all $i \in \Delta$, $0 \in P_i$. Hence, $0 \in \partial P$ is a
vertex.
According to Proposition \ref{art209} and Proposition \ref{art213}, if $P$ is the
polytope corresponding to an ample divisor, then all vertex $v$ of $P$ is precisely
the intersection of two facets $P_{j}$ and $P_{j+1}$ for some
$j \in \{0, \ldots, \, n-1 \}$. With this fact, we deduce that $-(K_X + D)$ is not
ample.
\end{proof}

\begin{prop}\label{art217}
Let $X$ be a complete smooth toric surface such that $\rk \Pic(X) \geq 3 \,$. If
$\card(\Delta) \in \{1, \, 2 \}$, then $-(K_X + D)$ is not ample.
\end{prop}

\begin{proof}
We start with the case $\card \Delta = 1$. We assume that $D = D_0 \,$. We have
$$
-(K_X + D) \cdot D_{n-1} = 1 - \gamma_{n-1} \quad \text{and} \quad
-(K_X + D) \cdot D_{1} = 1 - \gamma_1 ~~.
$$
If $-(K_{X} + D)$ is ample, then $\gamma_{n-1} \leq 0$ and $\gamma_1 \leq 0 \,$.
Let $A = \{\beta \, u_1 - \alpha \, u_0 : \, \alpha, \, \beta \geq 0 \}$,
$B = \{ -\alpha \, u_0 - \beta \, u_1 : \, \alpha, \, \beta \geq 0 \}$ and
$C = \{ \alpha \, u_0 - \beta \, u_1 : \, \alpha, \, \beta \geq 0 \}$.
As $X$ is complete, the fan $\Sigma$ of $X$ satisfies
$$
\bigcup_{\sigma \in \Sigma}{\sigma} = A \cup B \cup C \cup \Cone(u_0, \, u_1) ~.
$$
As $\gamma_1 = \det(u_0, \, u_2)$ and $\gamma_{n-1} = \det(u_{n-2}, \, u_0)$,
we deduce that $u_2 \in B$ and \linebreak $u_{n-2} \in A$.
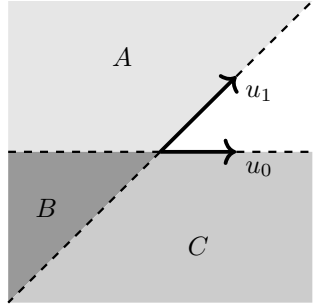
\begin{center}
\begin{tikzpicture}
\fill[color = gray!20] (2,2) --(0,0)--(-2, 0) -- (-2, 2) ;
\fill[color = gray!80] (-2, -2) -- (0,0)--(-2, 0) ;
\fill[color = gray!40] (-2, -2) --(0,0)--(2, 0) -- (2, -2) ;

\draw[line width = 0.3mm, dashed] (0, 0) -- (-2, 0) (1,0)--(2,0) ;
\draw[line width = 0.3mm, dashed] (0, 0) -- (-2, -2) (1,1)--(2,2) ;

\draw[line width = 0.5mm, ->] (0, 0) -- (1, 0) ;
\draw[line width=0.5mm, ->] (0, 0) -- (1, 1) ;

\draw (1.3, 0) node[below]{$u_0$} ;
\draw (1.3, 1) node[below]{$u_1$} ;
\draw (-0.5, 1.5) node[below]{$A$} ;
\draw (-1.5, -0.5) node[below]{$B$} ;
\draw (0.5, -1) node[below]{$C$} ;
\end{tikzpicture}
\end{center}
When $n \geq 5$, this is in
contradiction with the fact that if $u_2 \in B$, then $u_{n-2}$ is in $B$ or $C$.
Thus, we deduce that $-(K_X + D)$ is not ample.
\\[0.25cm]
We now assume that $\card(\Delta) = 2$. After renumbering the indices, we can assume
that $D = D_1 + D_j$ with $j \in \{2, \ldots, \, n-1 \}$.\\
{\it First case.} We assume that $j \in \{3, \ldots, \, n-1 \}$. Let $P$ be the
polytope of $-(K_X + D)$. As $0 \in P_1$ and $0 \in P_j$, we deduce that $0$ is a vertex
of $P$. Hence, for $k \in \{2, \ldots, \, j-1 \}$, we have $\vol(P_k) = 0$.
By Propositions \ref{art209} and \ref{art213}, we deduce that $-(K_X + D)$ is not ample.
\\[0.25cm]
{\it Second case.} We assume that $D = D_1 + D_2$. We have
$$
-(K_X + D) \cdot D_3 = 1 - \gamma_3 \quad \text{and} \quad
-(K_X + D) \cdot D_0 = 1 - \gamma_0 ~.
$$
Let $A = \{ -\alpha \, u_1 + \beta \, u_2 : \, \alpha, \, \beta \geq 0 \}$,
$B = \{ -\alpha \, u_1 - \beta \, u_2 : \, \alpha, \, \beta \geq 0 \}$ and
$C = \{ \alpha \, u_1 - \beta \, u_2 : \, \alpha, \, \beta \geq 0 \}$.
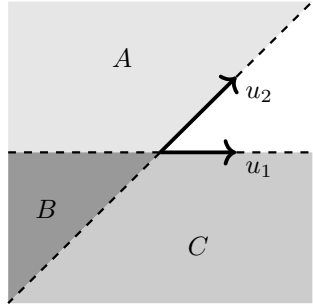
\begin{center}
\begin{tikzpicture}
\fill[color = gray!20] (2,2) --(0,0)--(-2, 0) -- (-2, 2) ;
\fill[color = gray!80] (-2, -2) -- (0,0)--(-2, 0) ;
\fill[color = gray!40] (-2, -2) --(0,0)--(2, 0) -- (2, -2) ;

\draw[line width = 0.3mm, dashed] (0, 0) -- (-2, 0) (1,0)--(2,0) ;
\draw[line width = 0.3mm, dashed] (0, 0) -- (-2, -2) (1,1)--(2,2) ;

\draw[line width = 0.5mm, ->] (0, 0) -- (1, 0) ;
\draw[line width=0.5mm, ->] (0, 0) -- (1, 1) ;

\draw (1.3, 0) node[below]{$u_1$} ;
\draw (1.3, 1) node[below]{$u_2$} ;
\draw (-0.5, 1.5) node[below]{$A$} ;
\draw (-1.5, -0.5) node[below]{$B$} ;
\draw (0.5, -1) node[below]{$C$} ;
\end{tikzpicture}
\end{center}
If $-(K_X + D)$ is ample, then $\gamma_3 \leq 0$ and $\gamma_0 \leq 0$.
As $\gamma_3 = \det(u_2, \, u_4)$ and $\gamma_{0} = \det(u_{n-1}, \, u_1)$,
we deduce that $u_4 \in C$ and $u_{n-1} \in A$.
If $n \geq 6$, this situation contradicts the positioning order of vectors $u_i$.
If $n = 5$, we have $u_4 \in A$ and $u_4 \in C$, this is not possible.
\\
Thus, we deduce that $-(K_X + D)$ is not ample.
\end{proof}

If $\Delta \neq \vide$, according to Propositions \ref{art216} and \ref{art217}, it is
enough to study the positivity of $-(K_X + D)$ when $\rk \Pic(X) \in \{1, \, 2 \}$ and
$\card(\Delta) \in \{1, \, 2 \} \,$.
The Propositions \ref{art216} and \ref{art217} prove Theorem \ref{art201}.

\subsection{Proofs of Theorems \ref{art202} and \ref{art203}}
Let $(e_1, \, e_2)$ be the standard basis of $\Z^2$.
The rays of the fan of $\Pm^2$ are the half-line generated by $u_1 = e_1, \, u_2 = e_2$
and $u_0 = -(e_1+e_2)$. The divisors $D_i$ of $\Pm^2$ defined in Theorem \ref{art202}
correspond to the divisors associated to the ray $\Cone(u_i)$.

\begin{proof}[Proof of Theorem \ref{art202}]
We have the linear equivalence $D_1 \sim D_0$ and $D_2 \sim D_0$. By Theorem \ref{art205},
we have $K_X = -(D_0 + D_1 + D_2)$, i.e $-K_X \sim 3 \, D_0$. As $D_0$ is ample,
we deduce that $-(K_X + D)$ is not ample if and only if $D = D_1 + D_2 + D_3$.
\end{proof}

$\,$\\[-0.25cm]
We now assume that $X = \Pm \left( \Oc_{\Pm^1} \oplus \Oc_{\Pm^1}(r) \right)$ with
$r \in \N$.
The rays of the fan of $X$ are the half lines generated by the vectors $u_1 = e_1$,
$u_2 = e_2$, $u_3 = -e_1 + r \, e_2$ and $u_0 = -e_2 \,$.
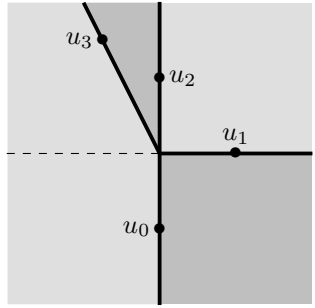
\begin{center}
\begin{tikzpicture}[scale=1]
\fill[gray!25] (0, 2) -- (0, 0) -- (2, 0) -- (2, 2);
\fill[gray!50] (0, -2) -- (0, 0) -- (2, 0) -- (2, -2) ;
\fill[gray!25] (-2, 2) -- (-2, -2) -- (0, -2) -- (0, 0) -- (-1, 2) ;
\fill[gray!50] (-1, 2) -- (0, 0) -- (0, 2) ;
\draw[line width = 0.5mm] (0, -2.01) -- (0, 2.01) ;
\draw[line width = 0.5mm] (0, 0) -- (2.01, 0) ;
\draw[line width = 0.5mm] (0, 0) -- (-1, 2) ;
\draw[dashed] (-2.01, 0) -- (0, 0) ;
\draw (1, 0) node{$\bullet$} node[above]{$u_1$} ;
\draw (0, 1) node{$\bullet$} node[right]{$u_2$} ;
\draw (-0.75, 1.5) node{$\bullet$} node[left]{$u_3$} ;
\draw (0, -1) node{$\bullet$} node[left]{$u_0$} ;
\end{tikzpicture}
\end{center}
The numbers $\gamma_i$ defined in (\ref{art207}) are given by
\begin{equation}\label{art218}
\gamma_0 = -r \qquad \gamma_1 = 0 \qquad \gamma_2 = r \qquad \gamma_3 = 0 ~~.
\end{equation}
By Proposition \ref{art209}, the divisor
$~ L = a_0 \, D_0 + a_1 \, D_1 + a_2 \, D_2 + a_3 \, D_3 ~$
is ample if and only if
$$
a_0 + a_2 >0 \quad , \quad
a_1 + a_3 > r \, a_2 \quad , \quad
a_1 + a_3 > - r \, a_0
$$
if and only if
\begin{equation}\label{art219}
a_0 + a_2 >0 \quad \text{and} \quad a_1 + a_3 > r \, a_2 ~.
\end{equation}
We have the linear equivalence of divisors
\begin{equation}\label{art220}
D_1 \sim D_3 \quad \text{and} \quad D_2 \sim D_0 - r \, D_3 ~.
\end{equation}
The divisor $D_i$ of $X$ corresponding the ray $\Cone(u_i)$ defined here correspond
to those defined in Theorem \ref{art203}.

\begin{proof}[Proof of Theorem \ref{art203}]
As $-K_X = D_0 + D_1 +D_2 + D_3$, by (\ref{art220}), we have
\begin{align*}
-K_{X} \sim & ~ 2 \, D_0 + (2-r)D_3 \\
-(K_X + D_0) \sim & ~ D_0 + (2-r)D_3 \\
-(K_X + D_3) \sim & ~ 2 \, D_0 + (1-r)D_3 \\
-(K_X + D_0 + D_3) \sim & ~ D_0 + (1-r)D_3
\end{align*}
and
\begin{align*}
-(K_X + D_2) \sim & ~ D_0 + 2 \, D_3 \\
-(K_X + D_2 + D_3) \sim & ~ D_0 + D_3 \\
-(K_X + D_2 + D_0) \sim & ~ 2 \,D_3 \\
-(K_X + D_1 + D_3) \sim & ~ 2 \,D_0 -r D_3
\end{align*}
If $a_1 = a_2 = 0$, the condition (\ref{art219}) becomes $a_0 >0$ and $a_3 >0$.
This allows us to conclude.
\end{proof}

\end{document}